\documentclass[12pt,a4]{amsart}

\usepackage{amssymb,mathrsfs,amsmath,dsfont,array}
\usepackage{multirow,arydshln}
\usepackage{tikz}
\usetikzlibrary{arrows,positioning,decorations.pathmorphing,shapes}
\usepackage{capt-of}
\usepackage{mathbbol}     
\usepackage{verbatim} 
\usepackage[all,cmtip]{xy}
\usepackage{amsbsy}
\usepackage{amsmath}
\usepackage{nccmath}

\usepackage{hyperref}
\usepackage[normalem]{ulem}

\newtheorem{Thm}{Theorem}[section]
\newtheorem{lem}[Thm]{Lemma}
\newtheorem{Pro}[Thm]{Proposition}
\newtheorem{cor}[Thm]{Corollary}

\newtheorem{rem}[Thm]{Remark}
\theoremstyle{definition}
\newtheorem{deft}[Thm]{Definition}
\newtheorem{exa}[Thm]{Example}

\newtheorem*{starredtheorem}{Theorem $\ddagger$}

\numberwithin{equation}{section}

\newcommand{\cQ}{\mathcal{Q}}

\renewcommand{\tilde}{\widetilde}

\def\a{\alpha}

\def\aa{\mathcal A}

\def\andd{\quad\hbox{and}\quad}

\def\b{\beta}

\def\bs{\boldsymbol}

\def\bl4{B_{\ell\geq4}}

\def\bbbc{{\mathbb C}}

\def\d{\delta}

\def\fg{\mathfrak{g}}

\def\fh{\mathfrak{h}}

\def\fk{\mathfrak{k}}

\def\lam{\lambda}

\def\LL{\mathcal{L}}

\def\fl{\mathfrak{L}}

\def\fm{(\cdot,\cdot)}

\def\bbbq{\mathbb{Q}}

\def\supp{\hbox{\rm supp}}

\def\1k{\frac{1}{k}}
\def\op{\oplus}
\def\ot{\otimes}

\def\la{\langle}
\def\ra{\rangle}

\def\sub{\subseteq}
\def\sg{\sigma}

\def\pf{\noindent{\bf Proof. }}

\def\sspan{\hbox{\rm span}}

\def\ft{\mathfrak{t}}

\def\bbbz{{\mathbb Z}}

\def\1il{1\leq i\leq\ell}

\newcommand{\summ}[1]{\raisebox{0.1ex}{\scalebox{1}{$\displaystyle \sum_{#1}\;$}}}
\newcommand{\Bigop}[2]{\raisebox{0.2ex}{\scalebox{0.7}{$\displaystyle \bigoplus_{#1}^{#2}\;$}}}

\newcommand{\blackdiamond}{\rotatebox[origin=c]{45}{$\vcenter{\hbox{$\scriptscriptstyle\blacksquare$}}$}}

\usepackage{lipsum}

 \title{Admissible modules over affine Lie superalgebras: The Final Step in the characterization}

\begin{document}

\maketitle

\pagestyle{myheadings}
\markboth{Admissible modules over affine Lie superalgebras}{M. Yousofzadeh}

\centerline{ Malihe Yousofzadeh}

\centerline{{\scalebox{0.65} {Department of Pure Mathematics, Faculty of Mathematics and Statistics, University of Isfahan,}}}\centerline{{\scalebox{0.65} { P.O.Box 81746-73441, Isfahan, Iran, and }}}

\centerline{{\scalebox{0.65} { School of Mathematics
 Institute for Research in
	Fundamental Sciences (IPM), }}}
\centerline{{\scalebox{0.65} {P.O. Box: 19395-5746, Tehran, Iran.}
}}

\centerline{{\scalebox{0.65} {		 ma.yousofzadeh@sci.ui.ac.ir \& ma.yousofzadeh@ipm.ir.	}}}

\curraddr{}


\begin{abstract}
Over the past three decades, there have been several attempts to characterize modules over affine Lie superalgebras. One of the main  issues in this regard is dealing with zero-level modules. In this paper, we  study these modules and {complete the characterization} of simple admissible  modules over affine Lie superalgebras. 
\end{abstract}
\maketitle

\section{Introduction}
\label{intro}
To explain the subject of this paper and present our results, we begin by reviewing  a few {well-known} definitions.  Suppose that  $\fk=\fk_0\op\fk_1$   is  a Lie superalgebra  having a root space decomposition with respect to 
 a finite dimensional subalgebra $\fh\sub \fk_0$ with corresponding root system $\Phi.$ We have $\Phi=\Phi_0\cup \Phi_1$ in which 
\[\Phi_i:=\{\a\in \Phi\mid \fk_i\cap \fk^\a\neq \{0\}\}\quad (i=0,1).\] Elements of $\Phi_0$ are called {even} and elements of $\Phi_1$ are called {odd.}
Suppose   that $\fk$ is equipped with a supersymmetric even non-degenerate invariant bilinear form which is non-degenerate on $\fh$ and transfer, naturally,  the form on $\fh$ to a symmetric bilinear form $\fm$ on $\fh^*.$
We say that a $\fk$-module $M$ is an $\fh$-{weight module} (or just weight module if there is no ambiguity) if $M$ is decomposed into weight spaces 
\[M^\lam:=\{v\in M\mid hv=\lam(h)v\quad (h\in\fh)\}\quad (\lam\in \fh^*).\]The set $\{\lam\in \fh^*\mid M^\lam\neq\{0\}\}$ is called the {support} of $M$.
For $T\sub \Phi$, we define
$T^{in}(M)$ (resp. $T^{ln}(M)$) to be the subset of $T$  consisting of all $\a\in T$ with $(\a,\a)\neq 0$ such that  $0\neq x\in\fk^\a$ acts on $M$ injectively (resp. locally nilpotently).
An  $\fh$-weight $\fk$-module  is called  a {finite}  $\fh$-{weight module} if each weight space is finite dimensional.  A finite weight $\fk$-module  $M$ is called integrable if $\{\a\in \Phi\mid (\a,\a)\neq 0\}=\Phi^{ln}(M)$. A finite weight $\fk$-module  $M$  is called  {admissible} if the multiplicity of weight spaces are uniformly bounded and  the support  of $M$  lies  in a single coset of $\fh^*/Q$ where $Q$ is the root lattice $\sspan_\bbbz \Phi$.

A subset $P$  of  $\Phi$  is called a {parabolic subset} of $\Phi$ if 
 ${\Phi}=P\cup -P$ and $(P+P)\cap {\Phi}\sub P.$ For a parabolic subset $P$ of $\Phi,$ we  have the decomposition
\[\fk=\fk^+_{_P}\op\fk^\circ_{_P}\op\fk^-_{_P}\] where
 $$\hbox{\small $\fk^\circ_{_P}:=\Bigop{\a\in P\cap -P}{}\fk^\a,\; \fk^+_{_P}:=\Bigop{\a\in P\setminus-P}{}\fk^\a\andd \fk^-_{_P}:=\Bigop{\a\in - P\setminus P}{}\fk^\a.$}$$
 Set   $\mathfrak{p}:=\fk^\circ_{_P}\op\fk^+_{_P}$. Each simple $\fk^\circ_{_P}$-module $N$ is a simple module over $\mathfrak{p}$ with the trivial action of $\fk^+.$ Then  $$\tilde N:=U(\fk)\ot_{U(\mathfrak{p})}N$$ is a $\fk$-module; here  $U(\fk)$ and $U(\mathfrak{p})$ denote, respectively,  the universal enveloping algebras of $\fk$ and $\mathfrak{p}.$ If the  $\fk$-module  $\tilde N$ contains a  unique maximal submodule $Z$ {intersecting} $N$ trivially and $P\neq \Phi$, the quotient module $${\rm Ind}^{\fk}_{P}(N):=\tilde N/Z$$ is called  a {parabolically induced} module. A simple  finite weight $\fk$-module  which is not parabolically induced is called {cuspidal}.

With these definitions in place, we now review the background of the subject. The study of modules over affine Lie (super)algebras has a long history, dating back to 1974, when V. Kac initiated the study of highest weight modules over Kac–Moody algebras.  An affine Lie (super)algebra $\fl$ contains a canonical central element $c.$ An $\fl$-module $M$ is said to be of level $\lam\in\bbbc$ if $c$ acts as  the scalar $\lam$ on $M$.  The level plays a fundamental role in the representation theory of $\fl.$
In  \cite{CP1}  and \cite{CP2}, V. Chari and A. Pressley investigated integrable simple objects in the category 
$\mathcal{C}$ of finite weight modules over affine Lie algebras, and showed that these are either highest–lowest weight modules or {the} so-called loop modules, {depending on whether the level is nonzero or zero}. Since then, several attempts have been made to study the simple objects of the category 
$\mathcal{C}$, as well as those of its subcategory consisting of admissible modules; see \cite{DG}  and the references therein. The same category 
$\mathcal{C}$ had been studied for reductive Lie algebras in \cite{F} and \cite{Mat}. In all these works, two key tools are parabolic induction and {so-called} twisted localization.

It was shown that a simple finite weight module over a reductive Lie algebra is parabolically induced from a cuspidal module \cite{F}. The concept of twisted localization was introduced in O. Mathieu’s work \cite{Mat} to study cuspidal modules over reductive Lie algebras. Using twisted localization, he proved that every cuspidal module over a reductive Lie algebra is isomorphic to a submodule of a certain module known as a coherent family, and then studied  these coherent families.

There are two main points in  Mathieu's characterization; first, the twisted localization process itself, and second, the study of coherent families.  In \cite{DG}, I.~Dimitrov and D.~Grantcharove  apply the twisted localization process for affine Lie algebras to reduce the study of cuspidal objects 
in the category $\mathcal{C}$ to the study of {some other simple objects obtained through this process.} 
The use of twisted localization  is  justified if the {requisite characterizations of these simple objects are available.} 
To address the problem of describing the simple constituents of twisted localizations for affine Lie algebras,  I.~Dimitrov and D.~Grantcharove make extensive use of constructions of Berman, Billig, Chari, and Pressley, as well as Mathieu’s classification; see \cite{BB, C, CP1, CP2, Mat}.

In this paper, we prove that twisted localization is also feasible in the super case and show how it completes the characterization problem. In our previous works, we identified the necessary characterizations for this purpose. We briefly summarize them below and, in doing so, introduce the notations needed for  the paper.

Suppose that $\fl$ is an affine Lie superalgebra with $\fl_1\neq \{0\}$ and assume  $\fh$ is the standard Cartan subalgebra of $\fl$ with corresponding root system $$R=R_0\cup R_1.$$ 
There is a non-degenerate bilinear form on $\fh$  which can be transferred naturally  to a bilinear form $\fm$ on the dual space $\fh^*$ of $\fh.$ The set of imaginary roots, namely $R_{im}:=\{\a\in R\mid (\a, R)=\{0\}\},$ is a free abelian group of rank one; say e.g., $\bbbz\d.$  For each subset $S$ of $\fh^*$, we define
 \[{\small\begin{array}{c}S_{re}:=\{\a\in S\mid (\a,\a)\neq 0\},\\
\hbox{\footnotesize{(real elements)}}\end{array}
~\begin{array}{c}S_{ns}:=\{\a\in S\setminus R_{im}\mid (\a,\a)= 0\}\\\hbox{\footnotesize{(non-singular elements)}}\end{array}\andd \begin{array}{c}S^\times:=S_{re}\cup S_{ns},\\\hbox{\footnotesize{(non-imaginary elements)}}.\end{array}}\] There is a {finite subset $\dot R$  of  $\fh^*$  which is a disjoint union $\dot R_0\uplus \dot R_1$ of two subsets $\dot R_0$ and $\dot R_1$ such that $\dot R_0\neq \{0\}$ is a finite root system with at most three irreducible components, $$R\sub \dot R+\bbbz\d\andd \dot R_1=\dot R_{ns}\cup  \{\dot\a\in\dot R_{re}\mid \exists m\in\bbbz~{\rm s.t}~ \dot\a+m\d,2(\dot\a+m\d)\in R\}.$$  Moreover,  for each $0\neq \dot\a\in \dot R,$ there are $r_{\dot\a}\in\{1,2,4\}$ and $0\leq p_{\dot\a}< r_{\dot\a}$  such that for $r:={\rm max}\{r_{\dot\a}\mid 0\neq \dot\a\in \dot R\},$  
 \begin{equation}\label{r}\tag{$\dagger$}
r_{\dot\a}\mid r,\quad
 r\bbbz\d+R\sub R\andd (\dot\a+\bbbz\d)\cap R=\dot\a+r_{\dot\a}\bbbz\d+p_{\dot\a}\d\quad(0\neq \dot\a\in \dot R);
\end{equation}
see \cite{van-thes} for the details.   
It is known} that if  $V$ is a simple finite weight $\fl$-module, then,  we have  
\[R_{re}=\{\a\in R\mid(\a,\a)\neq 0\}=R^{in}(V)\cup R^{ln}(V).\]

There are two important distinctions between affine Lie algebras and affine Lie superalgebras; the first  is the existence of non-singular roots and the second is the existence of real roots $\a$ such that  $2\a$ is also a root. These features make  the characterization problem significantly more complex in the superalgebra case, compared to the non-super case; e.g.,  {in both cases, whether  super or non-super,} $R_{im}\sub \sspan_\bbbz R^\times,$ {however} in the non-super case, $R^\times=R_{re}$ which in turn  implies that $R_{im}\sub\sspan_\bbbz R_{re};$ {the phenomena which it is not necessarily true  for affine Lie superalgebras}. We also know that $\fl^0$ contains a two-dimensional subspace $\bbbc c\op\bbbc d$ where $c$ is the canonical central element and $d$ is a degree derivation of $\fl,$ the quotient space  $\Bigop{k\in\bbbz}{}\fl^{k\d}$ modulo $\bbbc c\op\bbbc d$ is an algebra that arises in   the process of  parabolic induction. 
This algebra is an abelian Lie algebra in the non-super case, {whereas} in some  super case, it becomes a split central extension of a non-abelian Lie superalgebra {which} we refer to as a quadratic Lie superalgebra. The characterization problem, in particular, requires the study of modules over quadratic Lie superalgebras—a topic we have addressed in \cite{KRY}. {We also need  the following theorem}  which plays a pivotal role both in the development of the theory and in the present paper:

\begin{starredtheorem}[{\cite[Thm.~4.8]{you8}}]\label{property}
Suppose that $V$ is a simple finite $\fh$-weight $\fl$-module. Then, for each $\b\in R_{re},$ one of the following  happen:
\begin{itemize}
\item[\rm (i)]  $(\b+\bbbz\d)\cap R\sub R^{ln}(V),$
\item[\rm (ii)] $(\b+\bbbz\d)\cap R\sub R^{in}(V),$
\item[\rm (iii)] there exist $k\in\bbbz$ and $t\in\{0,1,-1\}$ such that for $\gamma:=\b+k\d,$
\begin{align*}
&(\gamma+\bbbz^{\geq 1}\d)\cap R\sub R^{in}(V),\quad (\gamma+\bbbz^{\leq 0}\d)\cap R\sub R^{ln}(V),\\
& (-\gamma+\bbbz^{\geq t}\d)\cap R  \sub R^{in}(V),\quad (-\gamma+\bbbz^{\leq t-1}\d)\cap R\sub R^{ln}(V),
\end{align*}
\item[\rm (iv)] there exist $k\in\bbbz$ and $t\in\{0,1,-1\}$ such that for $\eta:=\b+k\d,$
\begin{align*}
&(\eta+\bbbz^{\leq -1}\d)\cap R\sub R^{in}(V),\quad (\eta+\bbbz^{\geq 0}\d)\cap R\sub R^{ln}(V),\\
& (-\eta+\bbbz^{\leq -t}\d)\cap R  \sub R^{in}(V),\quad (-\eta+\bbbz^{\geq 1-t}\d)\cap R\sub R^{ln}(V).
\end{align*}
\end{itemize}
\end{starredtheorem}
{We call  $\dot\b\in\dot R_{re}$ and $\b\in R\cap (\dot\b+\bbbz\d)$ full-locally nilpotent,  full-injective,  down-nilpotent hybrid and up-nilpotent hybrid respectively if    for $\b,$ properties stated in  (i) up to (iv)  occurs respectively.  
The set of full-locally nilpotent roots, full-injective roots and hybrid roots are denoted by $R_{\rm f-ln}=R_{\rm f-ln}(V),$ $R_{\rm f-in}=R_{\rm f-in}(V)$ and $R_{\rm hyb}=R_{\rm hyb}(V)$ respectively.} 
\smallskip

For a simple finite weight $\fl$-module $V$, we have \[R_{re}=R_{\rm f-ln}(V)\cup R_{\rm f-in}(V)\cup R_{\rm hyb}(V).\] {This  helps  to divide} the category of finite weight $\fl$-modules into some subcategories and study the {corresponding simple  objects according to their instinct natures}.  
In what follows, we present a list of known results {which altogether} constitute prerequisites for the objectives of this paper. 
We have two main cases:
\[\bullet~ R_{re}=R_{\rm hyb}(V)\cup R_{\rm f-ln}(V) \quad \bullet~ R_{\rm f-in}(V)\neq \emptyset\] 
and the following subcases:

\noindent $\blackdiamond$ $R_{re}=R_{\rm hyb}(V):$ 
In this case, $V$ is parabolically induced from a cuspidal  module over a  finite dimensional Lie superlagebra $\fg=\op_{i=1}^n\fg(i)$ in which each $\fg(i)$ is either a basic classical simple Lie superalgebra or a reductive Lie algebra; see \cite{you8, DKYY}.\\

\noindent $\blackdiamond$ $R_{re}= R_{\rm f-ln}(V):$  In this case, the module $V$ is  integrable.
\begin{itemize}
\item[(i)] If $\dot R_0$ has  one simple component and $M$ {is of a nonzero level}, then, $V$ is a highest weight module.
\item[(ii)]  There is no integrable nonzero-level simple $\fl$-module  if $\dot R_0$ has  more that one simple component; see \cite{E1, you12}. 
\item[(iii)]   An integrable zero-level simple admissible  $\LL$-module  is parabolically induced from a finite weight  simple module over $\Bigop{m\in\bbbz}{}\fl^{m\d}.$ Moreover, $\Bigop{m\in\bbbz}{}\fl^{m\d}$ is a semi direct product  $H\rtimes D$ in which $H$ is a $\bbbz$-graded abelian Lie algebra if $\fl\neq A(2k,2l)^{(4)}$ while it is a split central extension of the  quadratic Lie superalgebra if $\fl=A(2k,2l)^{(4)}$, also  $D$ is a one dimensional degree derivation Lie algebra; see \cite{KRY, RZ, WZ}.
\end{itemize}
\noindent $\blackdiamond$ $R_{re}=R_{\rm hyb}(V)\cup R_{\rm f-ln}(V)$ and both $R_{re}\cap R_{\rm hyb}(V)$ and  $R_{re}\cap R_{\rm f-ln}(V)$
 are nonempty:  {These modules are a generalization of modules called quasi-integrable modules and their   characterization  is known from} \cite{you9, DKY}.\\

\noindent $\blackdiamond$ $R_{\rm f-in}(V)\neq \emptyset$ and $R_{\rm f-in}(V)\cap -R_{\rm f-in}(V)=\emptyset:$ If $V$ is admissible, then, $V$ is parabolically induced from a module which is either integrable or quasi-integrable; see \cite{you10}.\\

\noindent $\blackdiamond$ $R_{\rm f-in}(V)\cap -R_{\rm f-in}(V)\neq\emptyset:$ It is the objective  of this paper to show that the characterization of these modules, with bounded weight multiplicities,  is reduced to the characterization of modules $N$ for which $R_{\rm f-in}(N)\cap -R_{\rm f-in}(N)=\emptyset.$ We exclude types $X^{(1)}$ where $X$ is a basic classical simple  Lie superalgebra of type $I$; an interest reader can consult \cite{CFR} for these excluded types.

To obtain our results, we make use of  twisted localization. One of the fundamental requirement for using  the twisted localization process is that an admissible module contains a simple submodule. Mathieu \cite{Mat} established this result for finite-dimensional simple Lie algebras by showing that such modules are of finite length. However, in the case of affine Lie (super)algebras, admissible modules are not necessarily of finite length. For affine Lie algebras, I.~Dimitrov and D.~Grantcharov \cite{DG} verified this key property through extensive use of the classification of finite weight simple modules over affine Lie algebras. In this work, we prove this fact  in the super case using  only Theorem~$\ddagger$.  

{The paper contains four sections; following  this introduction, we have a short section devoted to introducing the super version of twisted localization process. Section~3 is devoted to proving the fact that each zero-level admissible module contains a simple submodule and in the last section, we apply twisted localization process to get our main result on the characterization of  simple finite weight modules $V$ with $R_{\rm f-in}(V)\cap -R_{\rm f-in}(V)\neq \emptyset.$}

\section{twisted localization}\label{admissible} 
In this short subsection, we follow \cite{Mat}  and  introduce the super analogue of  the notion of  {\it twisted localization}.
Suppose that $\fl=\fl_0\op\fl_1$ is a Lie superalgebra having a root space decomposition $\fl=\Bigop{\a\in \fh^*}{}\fl^\a$ with respect to a finite dimensional subalgebra $\fh$  of $\fl_0$ whose dual space is denoted by $\fh^*.$ Assume $R=R_0\cup R_1$ is the corresponding root system.

Assume ${{f}}\in \fl_0\cap\fl^\a$ is a nonzero root vector corresponding to a nonzero even root $\a$. 
Extend  $\{x_0:={{f}}\}$ to a base $\{x_i\mid i\in I\cup \{0\}\}$  of $\fl$ where $I$ is a totally ordered index set. If we extend the total ordering  on $I$ to a total ordering on $I\cup\{0\}$ such that $i<0$ for all $i\in I,$ we get form PBW Theorem that for $n\in\bbbz^{\geq0}$ and $a\in U:=U(\fl),$ where $U(\fl)$ stands for the universal enveloping algebra of $\fl,$ $a{f}^n=0$ if and only if $a=0.$ Also, if extend the total ordering on $I$ to a total ordering on $I\cup\{0\}$ such that $0<i$ for all $i\in I,$ again, we get form PBW Theorem that for $n\in\bbbz^{\geq0}$ and $a\in U=U(\fl),$ ${f}^na=0$ if and only if $a=0.$  In particular, $\mathscr{S}:=\{{f}^n\mid n\in \bbbz^{\geq 0}\}$ is both  right and left reversible multiplicative subset of $U$; see \cite{Lam}.
If moreover,  ${{f}}$ acts locally nilpotently on $\fl$, we will get  that ${\rm ad}({{f}})$ acts locally nilpotently on $U(\fl)$ which in turn implies that  for each $u\in U=U(\fl)$ and $n\in \bbbz^{\geq 0},$ there is $m\in \bbbz^{\geq 0}$ such that $u{f}^m\in {f}^n U(\fl);$ see \cite[Lem.~4.2]{Mat}. 
This means that $\mathscr{S}$ is right permutable; see \cite{Lam}. Therefore, we can form \[U_\a:=\mathscr{S}^{-1}U,\] and we can embed $U$ in $\mathscr{S}^{-1}U$.
Next, assume $M$ is a (left) $U$-module and set $D_\a(M):=U_\a\ot_U M.$ If  $M=\Bigop{\lam\in\fh^*}{}M^\lam$ is a weight module. Then, 
\[D_\a(M)=U_\a\ot_U M =\Bigop{n\in\bbbz^{\geq 0}}{}\Bigop{\lam\in{\supp}(M)}{}  f^{-n}\ot M^\lam\]
is a left $U_\a$-module. As a $U$-module (equivalently as an $\fl$-module), $D_\a(M)$ has a weight space decomposition with 
\[\supp(D_\a(M))\sub \supp(M)+\bbbz^{\leq0}\a.\] 
One can easily see the following:
\begin{itemize}
\item  If  ${\a}\in R^{in}(M)$, then $M$ can be embedded, as a $U$-module, in $D_\a(M)$ and  $\supp(D_\a(M))= \supp(M)+\bbbz^{\leq0}\a$.
\item If $-\a\in R_0$, $\pm \a\in R^{in}(M)$ and $M$ is a finite weight module, then 
$D_\a(M)\simeq M$ as $U$-modules.
 \item If $M$ is a $U_\a$-module, $D_\a(M)\simeq M$ as $U_\a$-modules. 
\end{itemize}

Next, suppose that $\Omega$ is a $U_\a$-module and assume $x\in\bbbc.$ Define 
\begin{align*}
\Theta_\a^x:&U_\a\longrightarrow U_\a\\
&u\mapsto \sum_{i=0}^\infty {x\choose i}{\rm ad}(f)^i(u)f^{-i}\quad\quad(u\in U_\a)
\end{align*}
where 
\[{x\choose 0}:=1\andd {x\choose i}:=\frac{x(x-1)\cdots(x-i+1)}{i!}\quad (x\in \bbbc,~i\in\bbbz^{\geq 1}).\]
The $U_\a$-module twisted by $\Theta_\a^x$ is denoted by $\Omega(x)$ and each element is correspondingly
denoted by $v(x)$. Also, for complex numbers $x$ and $y$, we have 
\begin{align*}
(\Omega(x))(y)=\Omega(x+y).
\end{align*}
If $\Omega$ is an $\fh$-weight $\fl$-module, then 
\[\supp(\Omega)-x\a=\supp(\Omega(x))\andd (\Omega(x))^{\lam-x\a}=\Omega^{\lam}\quad (\lam\in\supp(\Omega)).\]

\medskip

For an $\fl$-module  $M$, we denote $(D_\a(M))(x)$   by $D^x_\a(M).$  
If $M$ is an $\fh$-weight  $\fl$-module, so is  $D^x_\a(M)$ and we get
\[\supp(D^x_\a(M))\sub \supp(M)-x\a+\bbbz^{\leq 0}\a.\]  We also have the following:
\begin{itemize}
\item If $\a\in R^{in}(M),$ then $\supp(D^x_\a(M))= \supp(M)-x\a+\bbbz^{\leq0}\a.$
\item  If  $-\a\in R_0,$   $\pm\a\in R^{in}(M)$ and $M$ is a finite weight $U$-module, then, 
$\supp(D^x_\a(M))= \supp(M)+x\a$ and $$(D^x_\a(M))^{\lam-x\a}=M^{\lam}.$$
\item If $M$ is  a simple $U$-module and $\a\in R^{in}(M),$ $D_\a^x(M)$ is a  simple $U_\a$-module.
\end{itemize}
\begin{deft} Suppose that $M$ is a simple finite weight $U$-module. If there is $x\in \bbbc$, a simple finite weight $U$-module $N$ and a real {even} root $\a$   such that $M$ is isomorphic to  $D_\a^x(N),$ we say $M$ is a {\it twisted localization} of $N$ using $x$. 
\end{deft}

\section{{Finite weight modules over   affine Lie superelagebras}}
Our objective in this paper is  finalizing the characterization of simple  admissible   modules over  affine Lie superalgeras.  In this section, we pave the way for the next section which is devoted to  completing the characterization. Let us start with gathering some known facts on affine supertheory. We assume the reader is familiar with affine Lie superalgebras, however, for details on affine Lie superalgebras and their root systems, the reader may consult \cite{van-thes} and \cite[Exa.~3.1 \& 3.2, Table~5]{DSY}.  

Assume  $\fl$ is an affine Lie superalgebra and  $\fh$ is its standard Cartan subalgebra with corresponding root system $R=R_0\cup R_1$.  Recall $\dot R$ and $R_{re}$ from \S~\ref{intro}. In the sequel, we use the following notations:
 \begin{equation}\label{S}
\begin{array}{ll}Q_{re}:=\sspan_\bbbz R_{re},&
 \dot S:=\dot R\cap R_0\andd
 Q_0:=\sspan_\bbbz(R_0\cap(\dot S\setminus\{0\}+\bbbz\d)).
 \end{array}
 \end{equation}
We mention that 
\begin{equation}\label{} 
\parbox{4.7in}{if $\LL$   is neither of type $A(1,2\ell-1)^{(2)}$ nor of type  $X^{(1)}$ in which 
$X$ is {a type $I$ among the types} of basic classical simple Lie superalgebras,
both  $Q_0$ as well as $Q_{re}$ have finite indices
in $Q:=\sspan_\bbbz R$ because {the both span the finite dimensional $\bbbq$-vector space $\sspan_\bbbq R.$}}
\end{equation}

From now on until the end of  this paper, we assume $\fl$  is  an affine Lie superalgebra with $\fl_1\neq \{0\}$ {that  is neither of type $A(1,2\ell-1)^{(2)}$ nor of type  $X^{(1)}$ in which 
$X$ is  a type  $I$  basic classical simple Lie superalgebra}.  We assume $\fh$ is the standard Cartan subalgebra of $\fl$ with the corresponding root system $R=R_0\cup R_1$. {We keep the notations as introduced in \S~\ref{intro}}.

\begin{rem}\label{rem-new}
{\rm 
\begin{itemize}  Suppose $V$ is a simple finite weight  $\fl$-module.
\item[(i)]  If  $R^{in}(V)=R_{re},$ as $Q_{re}$ has the  finite index in $\sspan_\bbbq R,$ $V$  is 
{admissible}. 

\item[(ii)]  Assume  $\eta\in R_{re}$ and recall $r$ from (\ref{r}) in \S~\ref{intro}. Then, we have the following:
\medskip

\begin{itemize}
\item[(a)] 
If either $\eta,-\eta\in R_{\rm f-in}$ or $\eta$ is up-nilpotent hybrid, then,  
$\pm\eta-rm\d\sub R^{in}(V)$ for large enough $m$. Moreover, there is 
$\gamma\in R\cap(\eta+\bbbz\d)$ such that  $\gamma- r\d,-\gamma- r\d\in R^{in}(V);$ in particular,  $\supp(V)-2r\d\sub\supp(V).$
\item[(b)] 
If either $\eta,-\eta\in R_{\rm f-in}$ or $\eta$ is down-nilpotent hybrid, then,  
$\pm\eta+rm\d\sub R^{in}(V)$ for large enough $m$. Moreover, there is 
$\gamma\in R\cap(\eta+\bbbz\d)$ such that  $\gamma+ r\d,-\gamma+r\d\in R^{in}(V);$ in particular,  $\supp(V)+2r\d\sub\supp(V).$\end{itemize}
\end{itemize}}
\end{rem}

\begin{lem}\label{final} 
Suppose that $M$ is a simple finite weight $\fl$-module such that $T:=R_{\rm f-in}(M)\cap -R_{\rm f-in}(M)\neq \emptyset.$   Then, the level of $M$ is zero.
\end{lem}
\begin{proof}
Set \[T_0:= T \cap R_{0} ~ \hbox{ ~ as well as ~} ~ \dot\Psi=\{\dot{\alpha}\in {\dot R} \mid (\dot\a+\bbbz\d)\cap T_0\neq \emptyset\}.\] Then, $\dot \Psi$  is a finite root system; 
we note that $T_0\neq\emptyset$ as $2(R_{re}\cap R_1)\sub R_0.$  Pick an irreducible component  $\dot \Phi$ of $\dot\Psi$ and set 
$\Phi:=(\dot \Phi+\mathbb{Z}\delta)\cap R_0$. Then, up to some central space,  \[\fh+\summ{\alpha \in\Phi^\times}\fl^{\alpha}+\summ{\alpha,\b \in \Phi^\times}[\fl^\a,\fl^{\b}]\] is an affine Lie algebra $\fk$ whose canonical central element is the canonical central element $c$ of $\fl$; see \cite[Prop.~2.10(b)]{DKY}.
Pick a weight space $N$ of $M$ and form the $\fk$-submodule $W$ of $M$ generated by $N$. Then, $W$ is an admissible $\fk$-module
and  so by  \cite{BL3}, 
$c$ acts trivially on $W$ and so on $M$.
\end{proof}

\subsection{Admissible modules} 
The main goal of this subsection is proving the fact that an admissible $\fl$-module of level zero contains a simple submodule. We shall show this fact in Theorem~\ref{cont-simple}; 
an essential prerequisite for this theorem is Theorem~\ref{rank-unconfined} whose proof requires a series of  {lemmas}. Let us start with the following definition:

\begin{deft}\label{admissible-chain-M}
Set 
\[\dot S:=\dot R\cap R_0\andd \dot\fh=\sspan_\bbbc\dot S.\]
{\rm \item[(a)] A  weight $\fl$-module $L$ is called {\it confined} if  $\lam|{_{\dot\fh}}=0$ for each $\lam\in\supp(L)$, otherwise, we call it {\it un-confined}.
\item[(b)] Suppose that $M$ is a finite weight $\fl$-module.
\subitem(i)   An infinite chain \[C: ~~~~\cdots\sub N_2\subsetneq M_2\sub N_1\subsetneq M_1\sub M\] of submodules of $M$ is called an {\it admissible} chain if for each $i\geq 1,$ $M_i$ is generated by a nonzero weight vector and  $M_i/N_i$ is  a simple $\fl$-module.  The cardinality of the set \[\{i\in\bbbz^{\geq1}\mid M_i/N_i\hbox{ is un-confined}\}\] is called {\it un-confined rank} of $C$ and denoted by $R_{_C}.$
\subitem(ii) An admissible chain   \[C': ~~~~\cdots\sub N'_2\subsetneq M'_2\sub N'_1\subsetneq M'_1\sub M\] of submodules of $M$ is called a $k$-{\it overlaped} {\it tail-extension} of $C$  if $M_i=M'_i$ and $N_i=N'_i$ for all $1\leq i\leq k.$ 
\subitem(iii) The admissible chain $C$ is called {\it $k$-confined}  if  for $i>k,$ $M_i/N_i$ is confined.
\subitem(v) We say the admissible chain $C$ is  {\it tail-$k$-confined}  if  each $k$-overlapped tail-extension of $C$  is  $k$-confined.
\subitem}
\end{deft}
\begin{exa}\label{exa-first-m}
{\rm Suppose that $L$ is a  {weight  $\fl$-module}. If $R^{in}(L)\neq \emptyset,$ then $L$ is un-confined.}
\end{exa}

Throughout this section, until Theorem~\ref{rank-unconfined}, we assume $M$ is a zero-level admissible $\fl$-module, in particular, $\supp(M)$ lies in a single $Q$-coset where $Q$ is the root lattice $\sspan_\bbbz R$. Also, we suppose  $\mathfrak{b}\in\bbbz^{> 0}$ is a bound for the  weight multiplicities of $M$ and set 
\[d:=\mathfrak{b}+1.\] 
Moreover, we assume   \[C: \cdots\sub N_2\subsetneq M_2\sub N_1\subsetneq M_1\sub M\] is an admissible chain
and set
\[L_t:=M_t/N_t\quad\quad(t=1,2,\ldots).\] 
We have from Theorem~$\dag$ in \S~\ref{intro} that  for $t\in \bbbz^{\geq 1},$
{\small 
\begin{align}
R_{re}=  R_{\rm hyb}(L_t)~\cup~&(R_{\rm f-ln}(L_t)\cup -R_{\rm f-in}(L_t))~\cup~ (R_{\rm f-in}(L_t)\cup-R_{\rm f-ln}(L_t))\nonumber\\
~\cup~&  (R_{\rm f-ln}(L_t)\cup-R_{\rm f-ln}(L_t))~\cup~(R_{\rm f-in}(L_t)\cup -R_{\rm f-in}(L_t)).\label{decom-sec}\end{align} }
\begin{lem}\label{Step1}
Assume $\mu\in\supp(M)$. Then, there exist at most $d=\mathfrak{b}+1$ indices   $i\in\bbbz^{\geq 1}$ with $\mu\in\supp(L_i).$
\end{lem}
\begin{proof} 
To the contrary, assume $\mu\in \supp(M_{i_j}/N_{i_j})$ for $j=1,\ldots,\mathfrak{b}+2$ with $i_1<\cdots<i_{\mathfrak{b}+2}.$ This shows that  $\mu\in \supp(M_{i_j})$ for $j=1,\ldots,\mathfrak{b}+2$ and so $\mu\in \supp(M_{i_j})\cap \supp(N_{i_j})$ for $j=1,\ldots,\mathfrak{b}+1$. Therefore,  we have 
 {\small \[ \dim(N_{i_{\mathfrak{b}+1}}^\mu)\lneq\dim(M_{i_{\mathfrak{b}+1}}^\mu)\leq \dim(N_{i_{\mathfrak{b}}}^\mu)\lneq\dim(M_{i_{\mathfrak{b}}}^\mu)\cdots\leq  \dim(N_{i_{1}}^\mu)\lneq\dim(M_{i_{1}}^\mu)\leq \dim(M^\mu).\]}
This implies that $\dim(M^\mu)\geq \mathfrak{b}+1>\mathfrak{b}$ which is a contradiction. 
\end{proof}
\begin{lem}\label{fin-12}
Recall (\ref{S}) and suppose that {$\mu\in \fh^*.$}
\begin{itemize}
\item[(i)] 
Let $\dot R\setminus\{0\}=\{\dot\a_1,\ldots,\dot \a_p\}$ and fix $\dot\eta_{i}\in \dot S$ $(1\leq i\leq p)$ {with
$\frac{2(\dot\a_i,\dot\eta_i)}{(\dot\eta_i,\dot\eta_i)}\in\bbbz^{>0}.$} Recall $r$ from (\ref{r}) in \S~\ref{intro}.  There are at most    $\sum_{j=1}^pdr\frac{2(\dot\a_j,\dot\eta_j)}{(\dot\eta_j,\dot\eta_j)}$  (which is independent of $C$) indices $t$  
such that $L_t$ is an un-confined integrable module and  $\supp(L_t)\cap(\mu+Q_0)$ contains an element  whose restriction to  $\dot\fh $ is zero.
\item[(ii)] Assume $q$ is a nonzero weight for the {semisimple} Lie algebra
\[{\dot\fg:=\Bigop{0\neq \a\in \dot S}{} \fl^\a\bigoplus\sum_{0\neq \a\in \dot S}[ \fl^\a, \fl^{-\a}]}.\] 
 and  fix  $\dot\gamma\in\dot S$  with
$\frac{2(q,\dot\gamma)}{(\dot\gamma,\dot\gamma)}\in\bbbz^{>0}.$
There are at most $dr\frac{2(q,\dot\gamma)}{(\dot\gamma,\dot\gamma)}$ (which is independent of $C$) indices $t$ such that   $L_t$ is an un-confined integrable module and $\supp(L_t)\cap(\mu+Q_0)$ contains an element whose restriction to ${\dot\fh}$ is $q.$ 
\end{itemize}
\end{lem}
\pf (i)  To the contrary, assume there is   a  subset $T\subset \bbbz^{\geq1}$ of cardinality  $1+\sum_{j=1}^pdr\frac{2(\dot\a_j,\dot\eta_j)}{(\dot\eta_j,\dot\eta_j)}$  such that for $t\in T,$ there is  $\lam_t\in \supp(L_t)\cap (\mu+Q_0)$ with $\lam_t|_{_{\dot{\fh}}}=0.$ 
 This implies that   there is ${q_*\in \sspan_\bbbz \dot S}$ such that  
 \begin{equation}\label{zero-MM}\mu+q_*|_{_{\dot\fh}}=0\andd \lam_t\in \mu+q_*+\bbbz\d\quad(t\in T).\end{equation}
 Since for all $t\in T,$ $L_t$ is un-confined, its support 
 contains some weight whose restriction to $\dot{\fh}$ is nonzero. So, 
 there is {$\a\in R^\times\sub (\dot R\setminus\{0\})+\bbbz\d$} such that $\LL^\a L_t^{\lam_t}\neq \{0\}.$ 
 So, there is $1\leq j\leq p$  and $T'\sub T$ of cardinality $1+dr\frac{2(\dot\a_j,\dot\eta_j)}{(\dot\eta_j,\dot\eta_j)}$  such that for $t\in T'$, $\mu+\dot \a_j+q_*+m_t\d\in\supp(L_t)$ for some $m_t\in\bbbz.$ 
Since   $R^{ln}(L_t)=R_{re}$ and
\[
\dot\eta_j+2r\bbbz\d,\dot\eta_j+r\bbbz\d\sub R,
\]
 we  get that  the reflections $\sg_{\dot\eta_j+2rm\d}$ and $\sg_{\dot\eta_j+rm\d}$\footnote{$\sg_\eta(\lam)=\lam-2\frac{(\lam,\eta)}{(\eta,\eta)}\eta$, for $\eta\in R_{re}$ and $\lam\in\fh^*$.}, for $m\in\bbbz,$  preserve the support of $L_t.$ In particular,  since $(\d,\mu)=0,$ for  $m\in\bbbz$ and $t\in T'$, using (\ref{zero-MM}), we have 
\[\mu+q_*+\dot\a_j+m_t\d{+}\frac{2(\mu+q_*+\dot\a_j,\dot\eta_j)}{(\dot\eta_j,\dot \eta_j)}rm\d=\sg_{\dot\eta_j+2rm\d}\sg_{\dot\eta_j+rm\d}(\mu+q_*+\dot\a_j+m_t\d)\in\supp(L_{t}). \]
This implies that  for some $0\leq k_t<\frac{2(\dot\a_j,\dot\eta_j)}{(\dot\eta_j,\dot\eta_j)}r= \frac{2(\mu+q_*+\dot\a_j,\dot\eta_j)}{(\dot\eta_j,\dot\eta_j)}r,$  we have 
 \[\mu+q_*+\dot\a_j+k_t\d\in \supp(L_t)\quad(t\in T').\]
But $T'$ is of cardinality $1+dr\frac{2(\dot\a_j,\dot\eta_j)}{(\dot\eta_j,\dot\eta_j)}$. So, at least for $d+1$  indices $t$, $k_t$'s are equal, that is at least for $d+1$  indices $t$, $L_t$'s have the same weight $\mu+q_*+\dot\a_j+k_t\d.$ This  is a contradiction due to Lemma~\ref{Step1}.

(ii)  Assume to the contrary that  there is  a subset $T\sub \bbbz^{\geq 1}$ of cardinality  $1+ dr\frac{2(q,\dot\gamma)}{(\dot\gamma,\dot\gamma)}$ such that  for each $t\in T$, there is  $\lam_t\in \supp(L_t)\cap (\mu+Q_0)$ with $\lam_t|_{_{\dot{\fh}}}=q.$ 
 Therefore,  there is $q_*\in \sspan_\bbbz\dot S$ with $\mu+q_*\mid_{_{\dot{\fh}}}=q$ and 
$\lam_t\in \mu+q_*+m_t\d$ for some $m_t\in\bbbz$  ($t\in T$). As above
for all $m\in\bbbz$ and $t\in T,$ we have 
{\small \begin{align*}
 \mu+q_*+m_t\d{+}\frac{2(q,\dot\gamma)}{(\dot\gamma,\dot \gamma)}rm\d=&\mu+q_*+m_t\d{+}\frac{2(\mu+q_*,\dot\gamma)}{(\dot\gamma,\dot \gamma)}rm\d\\=&\sg_{\dot\gamma+2rm\d}\sg_{\dot\gamma+rm\d}(\mu+q_*+m_t\d)\in\supp(L_{t}),
 \end{align*}}
and so, using the same argument as in part~(i), we get a contradiction.\qed
\medskip

\begin{lem}\label{fi-10} 
Use  (\ref{decom-sec}) and consider the decomposition
\begin{equation*}\label{decom-*}\dot S\setminus\{0\}=\dot R\cap R_0\setminus\{0\}=X^{^t}_1\uplus-X^{^t}_1\uplus X_2^{^t}\uplus X^{^t}_3\uplus X^{^t}_4
\end{equation*} with 
\begin{equation*}\label{sub*}\begin{array}{ll}
R_0\cap (X^{^t}_1+\bbbz\d)\sub R_{\rm f-in}(L_t)\cap -R_{\rm f-ln}(L_t),&
R_0\cap (X^{^t}_2+\bbbz\d)\sub R_{\rm f-ln}(L_t)\cap -R_{\rm f-ln}(L_t),\\
R_0\cap (X^{^t}_3+\bbbz\d)\sub R_{\rm f-in}(L_t)\cap -R_{\rm f-in}(L_t),&
R_0\cap (X^{^t}_4+\bbbz\d)\sub R_{\rm hyb}(L_t).
\end{array}
\end{equation*}
Assume $\mu\in \fh^*.$ Then, we have the following:
\begin{itemize}
\item[(i)] For each  $t$ with $\supp(L_t)\cap (\mu+Q_0)\neq \emptyset$, there are $q^t_i\in\sspan_{\bbbz}({R_0}\cap (X^{^t}_i+\bbbz\d))$ $(i=1,2)$ and $k(t)\in r\bbbz$ (see ~(\ref{r}) in \S~\ref{intro}) such that $\mu+q^t_1+q^t_2+k(t)\d\in\supp(L_t).$
\item[(ii)] Set\footnote{We denote the cardinality of a set $X$  by $|X|$ and we say $X\sub R$ is closed if  $(X+X)\cap R\sub X.$} $P:=\{\emptyset\neq Y\sub\dot S\setminus\{0\}\mid Y=-Y {\hbox{ and $Y\cup\{0\}$ is closed}}\}.$
There are at most $2d|P|$ (which is independent of $C$) indices ${t}$ for which \[{X^{^t}_3\neq \emptyset},~ {X^{^t}_2=X^{^t}_1=\emptyset}\andd \supp(L_t)\cap (\mu+Q_0)\neq\emptyset.\]
\item[(iii)] Set $P':=\{\emptyset\neq Y\sub \dot S\setminus\{0\}\mid Y\cap-Y=\emptyset\}.$
There are at most $d|P'|$ indices (which is independent of $C$) $t$  for which \[ X^{^t}_2=\emptyset, X^{^t}_1\not=\emptyset\andd \supp(L_t)\cap (\mu+Q_0)\neq\emptyset.\]
\item[(iv)] 
For $X\in P$, denote by $m_{_X}$, the number of weights of the finite dimensional   semisimple Lie algebra \[\fk_{_{X}}:=\Bigop{ \dot\a\in X}{}\fl^{\dot\a}\op\sum_{\dot\a\in X}[\fl^{\dot\a},\fl^{-\dot\a}]\] which are either zero or a minuscule weight.
Then, the number of  indices $t$
 for which $X^{^t}_2$ is a nonempty proper subset of $\dot S\setminus\{0\}$ and $\supp(L_t)\cap (\mu+Q_0)\neq\emptyset$ is  at most  $\displaystyle{(|P'|+1)rd\sum_{\dot S\setminus\{0\}\neq X\in P}m_{_{X}}}$ (which is independent of $C$). \end{itemize}
\end{lem}
\pf (i)
For each  $t$ with $\supp(L_t)\cap (\mu+Q_0)\neq \emptyset$, there are $q^t_i\in\sspan_{\bbbz}({R_0}\cap (X^{^t}_i+\bbbz\d))$ $(i=1,2,3,4),$ where  $q^t_i=0$ if $X_i^t=\emptyset$,  with 
\[\mu+q^t_1+q^t_2+q^t_3+q^t_4\in\supp(L_t).\] Since ${R_0}\cap (X^{^t}_3+\bbbz\d)\sub R_{\rm f-in}(L_t)\cap -R_{\rm f-in}(L_t)$, we get that 
$\mu+q^t_1+q^t_2+q^t_4\in\supp(L_t).$ Also, as $R_0\cap (X^{^t}_4+\bbbz\d)\sub R_{\rm hyb}(L_t),$ we find $k(t)\in r\bbbz$ such that $\supp(L_t)-q^t_4+k(t)\d\sub\supp(L_t),$
in particular, 
$\mu+q^t_1+q^t_2+k(t)\d\in\supp(L_t)$ as we desired.

(ii) It is enough to show that for a given nonempty  symmetric  subset $X$ (i.e., $X=-X$) of {$\dot S\setminus\{0\}$ such that $X\cup\{0\}$ is closed,}
 there are at most $2d$ indices $t$ with \[\supp(L_t)\cap (\mu+Q_0)\neq \emptyset,~~X^{^t}_2=X^{^t}_1=\emptyset\andd X^{^t}_3=X.\]
To the contrary, assume we have at least  $2d+1$ such  indices  $t_1,\ldots, t_{2d+1}.$ Keep the same notation as in part~(i), since  for $t=t_1,\ldots, t_{2d+1}$, $X^{^t}_2=X^{^t}_1=\emptyset,$  we have 
 $q^t_1=q^t_2=0$. 
Therefore, we have \[\mu+k(t)\d\in\supp(L_t)\quad (t=t_1,\ldots, t_{2d+1}).\]
Since $X^{^t}_3=X\neq \emptyset$ for $t=t_1,\ldots,t_{2d+1},$  using Remark~\ref{rem-new}(ii), we get  that for each $t=t_1,\ldots,t_{2d+1},$ either $\mu+\d\in \supp(L_t)$ or $\mu\in \supp(L_t),$ depending on $k(t)$ is odd or even. This implies that either $\mu$ or $\mu+\d$ belongs to at least $d+1$ modules $L_t.$ This   contradicts Lemma~\ref{Step1}.

(iii) 
It is enough to show that for a given nonempty subset $X$ of $\dot S\setminus\{0\}$ with $X\cap-X=\emptyset,$ there are at most $d$ indices $t$ with \[\supp(L_t)\cap (\mu+Q_0)\neq \emptyset,~~X^{^t}_2=\emptyset\andd X^{^t}_1=X.\] To the contrary, assume we have at least  $d+1$ such  indices; say e.g., $t_1,\ldots, t_{d+1}.$ For $t=t_1,\ldots, t_{d+1},$ since $X^{^t}_2=\emptyset,$  we have 
 $q^t_2=0$. So, for $k(t)$ as in  part~(i), we have \[\mu+k(t)\d+q^t_1\in\supp(L_t)\quad (t=t_1,\ldots,t_{d+1}).\]  
Assume \[X=\{\dot\b_1,\ldots,\dot\b_{s}\}.\] Since $q_1^{t}\in \sspan_\bbbz (R_0\cap (X+\bbbz\d))$ and $X\sub \dot S\setminus\{0\}\sub R,$ recalling (\ref{r}) in \S~\ref{intro}, we have $\dot\b_i+r_{\dot\b_i}\bbbz\d\sub R$ for $1\leq i\leq s,$ and that 
 there are 
$n_i^t, m_i^t \in\bbbz$ ($1\leq i\leq {s}$) such that \[q_1^t=m_1^t(\dot\b_1+r_{\dot\b_1}n_1^t\d)+\cdots+m_s^t(\dot\b_s+r_{\dot\b_s}n_s^t\d).\]  So,  \[\mu+k(t)\d+m_1^t(\dot\b_1+r_{\dot\b_1}n_1^t\d)+\cdots+m_s^t(\dot\b_s+r_{\dot\b_s}n_s^t\d)\in\supp(L_t).\] Since $\dot\b_1-k(t)\d-r_{\dot\b_1}m_1^tn_1^t,\dot\b_i-r_{\dot\b_i}m_i^tn_i^t\in R^{in}(L_t)$ for $2\leq i\leq s,$ this implies that 
 \[\mu+((m_1^t+1)\dot\b_1+\cdots+(m^t_{s}+1)\dot\b_{s})\in\supp(L_t)\quad(t=t_1,\ldots,t_{d+1}).\]
As {$\dot\b_1,\ldots,\dot\b_{s}\in R^{in}(L_t),$} for $m:={\rm max}\{|m_1^t+1|,\ldots,|m^t_{s}+1|\mid  t=t_1,\ldots,t_{d+1}\},$ we have 
\[\mu+(m\dot\b_1+\cdots+m\dot\b_{s})\in \supp(L_t)\quad(t=t_1,\ldots,t_{d+1}),\] which contradicts Lemma~\ref{Step1}. This completes the proof.

\smallskip

(iv)  Since $X^{^t}_2$ is a proper nonempty symmetric subset of $\dot S\setminus\{0\}$ such that  $X^{^t}_2\cup \{0\}$ is closed,  it is enough to assume 
$X$ is a proper nonempty symmetric   subset of $\dot S\setminus\{0\}$ such that  $X\cup \{0\}$ is closed and show that  there are at most $(1+|P'|)rdm_{_X}$ indices $t$ with $X^{^t}_2=X.$
Assume 
$X=\{\dot\theta_1,\ldots,\dot\theta_{{l}}\}$.  Since $\mu+k(t)\d+q_1^t+q_2^t\in\supp(L_t)$ and $q_2^t\in\sspan_\bbbz(R_0\cap (X+\bbbz\d))$, there are 
\[ \theta_{i}\in R_0\cap (\dot\theta_{i}+\bbbz\d)\andd n_{i}^t\in\bbbz \quad (1\leq i\leq {l})\] 
such  that 
 \[\mu+k(t)\d+q_1^t+(n_1^t\theta_1+\cdots+n^t_{{l}}\theta_{{l}})\in\supp(L_t).\] 
 Assume $n(t)$ is the coefficient of $\d$ in  $n_1^t\theta_1+\cdots+n^t_{{l}}\theta_{{l}}$. There are  $p(t)\in r\bbbz$  as well as $0\leq r_0(t)<r$ with  $ n(t)=p(t)+r_0(t)$. 
 We have  
 \begin{equation}\label{remainder}
 \mu+(k(t)+p(t))\d+q_1^t+(n_1^t\dot\theta_1+\cdots+n^t_{{l}}\dot\theta_{{l}})+r_0(t)\d\in\supp(L_t).
 \end{equation}
 Since $X^{^t}_2=X$ is a nonempty proper subset of $\dot S\setminus\{0\},$ either $X_1^{^t}\not=\emptyset$ or  $X_1^{^t}=\emptyset$ and $X_4^{^t}\cup X_3^{^t}\not=\emptyset.$
We shall show that the cardinalities of the sets 
\[\{t\mid X^{^t}_2=X,~~X_1^{^t}=Y \}\andd \{t\mid X^{^t}_2=X,~~X_1^{^t}=\emptyset,X_4^{^t}\cup X_3^{^t}\not=\emptyset\}\]
for each  nonempty {proper} subset $Y$ of $\dot S\setminus\{0\}$ with $Y\cap-Y=\emptyset$, 
are at most $rdm_{_X}$ and $2rdm_{_X}$ respectively.

\noindent {\bf Case~1.} $\bs{X_1^{^t}=Y}$ {\bf where} $\bs{Y}$ {\bf is a nonempty {proper} subset  of $\bs{\dot S\setminus\{0\}}$ with} $\bs{Y\cap-Y=\emptyset}$: To the contrary, assume there is a subset $T$ of cardinality $rdm_{_X}+1$ such that for $t\in T,$ $X_1^{^t}=Y$ and $X_2^{^t}=X$.  Assume $Y=\{\dot\b_1,\ldots,\dot\b_{{s}}\}.$ As above, 
\[q_1^t=m_1^t(\dot\b_1+r_{\dot\b_1}n_1^t\d)+\cdots+m_s^t(\dot\b_s+r_{\dot\b_s}n_s^t\d)\]  for some \[n_i^t, m_i^t \in\bbbz\quad (1\leq i\leq {s}).\] 
Recall (\ref{remainder}), since $\dot\b_1-(k(t)+p(t)+n_1^tm_1^tr_{\dot\b_1})\d, \dot\b_i-n_i^tm_i^tr_{\dot\b_i}\d\in R^{in}(L_t)$ for $1\leq i\le {s}$ and  $t\in T,$
we have 
 \[\mu+((m_1^t+1)\dot\b_1+\cdots+(m^t_{{s}}+1)\dot\b_{{s}})+(n_1^t\dot\theta_1+\cdots+n^t_{{l}}\dot\theta_{{l}})+r_0(t)\d\in\supp(L_t).\]
For $m:={\rm max}\{|m_1^t+1|,\ldots,|m^t_{{s}}+1|\mid t\in T\},$ since $\dot\b_i\in {R^{in}(L_t)}$ for $1\leq i\leq s$ and $t\in T,$ we have
\[\mu+(m \dot\b_1+\cdots+m\dot\b_{{s}})+(n_1^t\dot\theta_1+\cdots+n^t_{{l}}\dot\theta_{{l}})+r_0(t)\d\in\supp(L_t)\quad(t\in T).\]  But  cardinality  of $T$ is $rdm_{_X}+1,$ so, there is $0\leq r_0<r$ and a subset $T'$ of $T$ of cardinality $dm_{_X}+1$ such that for $t\in T'$, $r_0=r_0(t).$ Therefore, for 
\[ \Pi_t:=\{\mu+(m \dot\b_1+\cdots+m \dot\b_{{s}})+q+r_0\d\mid q\in \sspan_\bbbz X\},\]
we get that $\Bigop{\nu\in\Pi_t}{}L_t^{\nu}$ is an integrable  module over {finite dimensional semisimple} Lie algebra $\fk_{_{X}}=\Bigop{ \dot\a\in X}{}\fl^{\dot\a}\op\sum_{\dot\a\in X}[\fl^{\dot\a},\fl^{-\dot\a}]$ 
and so {it contains  a weight which is either zero or a minuscule weight.}  Since the cardinality of $T'$ is $dm_{_X}+1$,  there is $q_*\in\sspan_\bbbz X$ and indices $t_1,\ldots,t_{d+1}$ such that $\mu+(m\dot\b_1+\cdots+m\dot\b_{{s}})+q_*+r_0\d\in \Pi_t\cap \supp(L_t)$ for $t=t_1,\ldots,t_{d+1}$  which is a contradiction; see Lemma~\ref{Step1}.

\medskip

\noindent {\bf Case~2.} $\bs{X_1^{^t}=\emptyset$ \bf{and} $X_4^{^t}\cup X_3^{^t}\not=\emptyset}:$ To the contrary, assume there is  $T\sub\bbbz^{\geq 1}$ of cardinality $2rdm_{_X}+1$ such that for $t\in T,$  
 $X_1^{^t}=\emptyset$ and $X_4^{^t}\cup X_3^{^t}\not=\emptyset.$ Recalling (\ref{remainder}),
 there is $0\leq r_0<r$ and a subset $T'$ of $T$ of cardinality $2dm_{_X}+1$ such that for $t\in T'$, $r_0=r_0(t).$ 
 By Remark~\ref{rem-new}(ii), we get  \begin{equation}\label{new-new}
2r\d+\supp(L_t)\sub\supp(L_t)\hbox{ or } -2r\d+\supp(L_t)\sub\supp(L_t).
\end{equation} 
This together with (\ref{remainder}) and the fact that $q_1^{^t}=0$ ($t\in T$), implies that there is an integer $m$ such that for each $t\in T'$ either 
\[\mu+2m\d+(n_1^t\dot\theta_1+\cdots+n^t_{{l}}\dot\theta_{{l}})+r_0\d\in\supp(L_t)\] or 
\[\mu+(2m+1)\d+(n_1^t\dot\theta_1+\cdots+n^t_{{l}}\dot\theta_{{l}})+r_0\d\in\supp(L_t).\]  Therefore as the cardinality of $T'$ is  $2dm_{_X}+1$, there is a subset $T''$ of $T'$ of cardinality  $dm_{_X}+1$ such that either 
\[\mu+2m\d+(n_1^t\dot\theta_1+\cdots+n^t_{{l}}\dot\theta_{{l}})+r_0\d\in\supp(L_t)\quad(t\in T'')\] or 
\[\mu+(2m+1)\d+(n_1^t\dot\theta_1+\cdots+n^t_{{l}}\dot\theta_{{l}})+r_0\d\in\supp(L_t)\quad(t\in T'').\]
Now, using the same argument as in the previous case, we get  a  contradiction.
\qed

\begin{lem} \label{fin-13} Recall that  {$Q_{re}=\sspan_\bbbz R_{re}$} has the finite index in $Q=\sspan_\bbbz R$ and assume
$\supp(M)\sub\bigcup_{j=1}^{n_*}(\nu_j+Q_{re}).$
Then, there are  at most $4dn_*$ (which is independent of $C$) indices  $t$ such that $R_{re}=R_{\rm hyb}(L_t)$.
\end{lem}
\pf We know from \cite[Lem.~5.6]{you8} that if  $L_t$ ($t\in\bbbz^{\geq 1}$) is hybrid, it is either up-nilpotent hybrid or down-nilpotent hybrid. So, it is enough to  fix $1\leq j\leq n_*$ and  prove that   there are at most $2d$ indices $t\in\bbbz^{\geq 1}$ such that 
\begin{equation*}\label{1-1}
\supp(L_t)\cap (\nu_j+Q_{re})\neq \emptyset\andd (\dot R_{re}+\bbbz^{\ll 0}\d)\cap R\sub R^{in}(L_t)
\end{equation*} and that 
 there are at most $2d$ indices $t$ such that   
\begin{equation*}\label{1-2}
\supp(L_t)\cap (\nu_j+Q_{re})\neq \emptyset\andd (\dot R_{re}+\bbbz^{\gg 0}\d)\cap R\sub R^{in}(L_t).\end{equation*}
As the proofs of the  both cases are similar, we just prove the first one. To the contrary, assume there are distinct indices $i_1, i_2, \ldots, i_{2d+1}$ such that 
 \begin{equation}\label{New-MMM}(\dot R_{re}+\bbbz^{\ll 0}\d)\cap R\sub R^{in}(L_{i_t})\andd \supp(L_{i_t})\cap (\nu_j+Q_{re})\neq \emptyset\quad (t=1,\ldots,2d+1).\end{equation}
For each $1\leq t\leq 2d+1$, 
there are  real roots $\b^t_1,\ldots,\b^t_{s_t}$ such that 
\begin{equation}\label{1-3}
\nu_j+\b^t_1+\cdots+\b^t_{s_t}\in\supp(L_{i_t}).
\end{equation} We get from (\ref{r}) in \S~\ref{intro} together with (\ref{New-MMM}) that there are $k^t_1,\ldots,k^t_{s_t}\in \bbbz^{\geq 1}$ such that 
\[-\b_1^t-rk_1^t\d,\ldots,-\b_{s_t}^t-rk_{s_t}^t\d\in R^{in}(L_{i_t})\quad (t=1,\ldots,2d+1).
\] Therefore,  (\ref{1-3}) implies that for $k_t:=k^t_1+\cdots+k^t_{s_t},$ we have \[\nu_j-rk_t\d\in\supp(L_{i_t})\quad (t=1,\ldots,2d+1).\]
So, contemplating  Remark~\ref{rem-new}(ii), we get that  for large enough $k,$ either $\nu_j-2rk\d\in\supp(L_{i_t})$ or $\nu_j-r(2k+1)\d\in\supp(L_{i_t})$ $(t=1,\ldots,2d+1)$  depending on whether $k_t$ is even or odd respectively.  This  contradicts  Lemma~\ref{Step1} and so we are done.\qed
\begin{Thm}\label{rank-unconfined}  Assume $M$ is a zero-level admissible $\fl$-module. Then, there is a bound for the un-confined ranks of admissible chains of submodules of {$M$}; that is,
there is a positive integer $\frak{l}$ such that for each 
admissible chain $C$ of submodules of $M,$ $R_{_C}\leq \frak{l}$. 
\end{Thm}
\pf  Contemplating Example~\ref{exa-first-m}, we need to show that  there is a positive integer $\frak{l}$ such that 
for each admissible chain 
 \[C: \cdots\sub N_2\subsetneq M_2\sub N_1\subsetneq M_1\sub M\] of submodules of $M$ and 
\[L_t(C):=M_t/N_t\quad\quad(t=1,2,\ldots),\] the cardinality 
\[\{t\in\bbbz^{\geq 1}\mid R^{in}(L_t(C))\neq \emptyset\}\cup\{t\in\bbbz^{\geq 1}\mid R^{ln}(L_t(C))=R_{re} \hbox{ and $L_t(C)$ is un-confined}\} \] is at most $\frak{l}.$
Since $\supp(M)$ lies in a single $Q$-coset and  {$Q_0$ as well as $Q_{re}$} have  finite indices  in $Q$, 
 we have 
\[
\supp(M)\sub\bigcup_{j=1}^{m_*}(\mu_j+Q_0)\andd \supp(M)\sub\bigcup_{j=1}^{n_*}(\nu_j+Q_0)\]
in which $m_*$  and $n_*$ are positive integers and $\mu_j$'s as well as $\nu_j$'s are elements of $\fh^*.$ 
Using Lemmas~\ref{fi-10}~\&~\ref{fin-13}, for each $1\leq j\leq m,$ there is a positive integer $\frak{l}_j$ such that for each admissible chain $C$ of submodules of $M$, the cardinality  
\[\{t\in\bbbz^{\geq 1}\mid R^{in}(L_t(C))\neq \emptyset\}\] is at most $4dn_*+\sum_{j=1}^{m_*}\frak{l}_j.$
Next recall (\ref{S}) and set \[{\dot\fg:=\Bigop{0\neq \a\in \dot S}{} \fl^\a\Bigop{}{}\sum_{0\neq \a\in \dot S}[ \fl^\a, \fl^{-\a}]}.\] Assume \[\{q_0,\ldots,q_l\}\] is the set of weights of $\dot\fg$ which  are either zero or minuscule weight of $\dot\fg.$  
Now let, for some admissible chain $C$ of submodules of $M$ and a positive integer $t$, $L_t:=L_t(C)$ be an  integrable un-confined  module. We have 
\[\displaystyle{L_t=\Bigop{\mu\in\supp(M)}{}L_t^\mu=\sum_{q\in \sspan_{\bbbz}\dot S}\sum_{j=1}^{m_*}\sum_{\substack{m\in \bbbz\\
q+m\d\in Q_0}}L_t^{\mu_j+q+m\d}}.\] If $1\leq j\leq m_*$ and  $\displaystyle{\sum_{m\in \bbbz}\sum_{q\in \sspan_{\bbbz}\dot S}L_t^{\mu_j+q+m\d}}\neq \{0\}$ (that is, $\mu_j+Q_0$ intersects $\supp(L_t)$ nontrivially), it  is a summation of  integrable finite weight modules over $\dot\fg$; in particular, as a  $\dot\fg$-module, $L_t$ contains a weight $q\in\{q_i\mid 0\leq i\leq l\}$.
Using  Lemma~\ref{fin-12}, there is a positive integer $\frak{l}'_{i,j}$ (independent of $C$) such that the number of such $t$'s is at most $\frak{l}'_{i,j}$. This completes the proof.
\qed

\begin{cor}\label{cor--1}
 Suppose that $M$ is a zero-level  admissible $\fl$-module 
and that $M$ is equipped with an  admissible chain 
of submodules of $M$, then, there is  a positive integer $k$ and a  tail-$k$-confined admissible chain of submodules of $M$.
\end{cor}
 \begin{proof}
 Suppose that 
  \[C:\cdots\sub N_2\subsetneq M_2\sub N_1\subsetneq M_1\sub M\]  is an admissible chain of submodules of $M$.
 Recall $\frak{l}$ from Proposition~\ref{rank-unconfined} and set 
 \[i_{_C}:={\rm min}\{i\in\bbbz^{\geq 1}\mid \hbox{$M_j/N_j$ is confined for $j\geq i$}\}.\] Therefore,  $C$ is $i_{_C}$-confined. To carry out the proof, we use induction on $\frak{l}-R_{_C}$ where $R_{_C}$ denotes the un-confined rank of $C$. We first assume  $R_{_C}=\frak{l}.$ Then, by Proposition~\ref{rank-unconfined}, $C$ is tail-$i_{_C}$-confined  and so we are done. 
Next, assume $k:=\frak{l}-R_{_C}$ is a positive integer. If $C$ is tail-$i_{_C}$-confined, we are done, otherwise, there is a $i_{_C}$-overlapped tail-extension $C'$ of $C$   which is not $i_{_C}$-confined, this gives that $R_{_{C'}}>R_{_C}$ and so, we are done using the induction hypothesis.
\end{proof}

 \begin{Thm}\label{cont-simple}
  Suppose that $M$ is a zero-level admissible $\fl$-module. 
Then,
 $M$ contains a simple submodule.
 \end{Thm}
 \begin{proof}
 We first notice that 
 \begin{equation}\label{MM-2}
 \parbox{5.5in}{
 if $v\in M$ is a nonzero weight vector, then  $U(\fl)v$ contains a maximal submodule $Z$.  If  $Z=\{0\}$, we get that $U(\fl)v$ is a simple submodule of $M,$ otherwise, $Z$ contains a nonzero weight vector $w$ and $U(\fl) w$ contains a maximal submodule.}
 \end{equation} Using (\ref{MM-2}) repeatedly, by starting from a  nonzero weight vector $v_1\in M,$ we  get either a simple submodule, which completes the proof, or an admissible chain of submodules. In the last case,  by Corollary~\ref{cor--1}, there is a positive integer $i_*$ such that  $M$ is equipped with a tail-$i_*$-confined chain 
 \[C:\cdots\sub N_2\subsetneq M_2\sub N_1\subsetneq M_1\sub M\]  of submodules of $M$.
For each $i\geq 1,$ since $C$ is admissible, we have   $M_i=U(\fl)v_i$ for some   nonzero weight vector $v_i$ of weight $\mu_i$. We recall that $\dot\fh=\sspan_{\bbbc}(R_0\cap\dot R)$ and  mention that as $C$ is $i_*$-confined, we have \[\mu_{i_*+1}\mid_{\dot\fh}=0.\]  

We complete the proof through the following two cases:

$\bullet$ $\fl^\a v_{i_*+1}=\{0\}$ for all  $\a\in R^\times$: In this case, using Lemma~3.7(ii) of \cite{DSY}, we get that the submodule of $M$ generated by $v_{i_*+1}$ equals $\bbbc v_{i_*+1}$ and so it is simple. This completes the proof in this case.

$\bullet$ There is  $\b\in R^\times$ with  $\fl^\b v_{i_*+1}\not =\{0\}:$
We pick $0\neq w\in\fl^\b v_{i_*+1}\sub M^{\mu_{i_*+1}+\b}$. The submodule  $M'_{i_*+1}:=U(\fl)w\sub M_{i_*+1}\sub N_{i_*}$ contains a maximal submodule $N'_{i_*+1}.$  
 
Since $\mu_{i_*+1}+\b$ is a weight of $M'_{i_*+1}/N'_{i_*+1}$ and  satisfies $\mu_{i_*+1}+\b|_{\dot\fh}\neq 0$,  we get that $M'_{i_*+1}/N'_{i_*+1}$ is not confined.
Considering 
\[C': N'_{i_*+1}\subsetneq M'_{i_*+1}\sub N_{i_*}\subsetneq M_{i_*}\sub \cdots\sub N_2\subsetneq M_2\sub N_1\subsetneq M_1\sub M,\]
we conclude  that $ M'_{i_*+1}$  (and so $M$) contains a simple submodule as otherwise using  the chain $C'$ together with the fact (\ref{MM-2}), we obtain  an $i_*$-overlapped  tail-extension  of $C$ which is not $i_*$-confined due to the choice of $i_*.$
\end{proof}

\section{Twisted Localization and end of the characterization}\label{local}
Throughout this section, we keep the same notations as in the text.  The following proposition is the ``super"  version of a well-known result in the non-super case; see \cite[\S~4]{Mat} for the finite case and \cite[Thm.~3.35]{DG} for the affine case. 
{Although this proposition resembles the non-super case, its central novelty lies in the differences in the proof that every  admissible module with trivial action of the canonical central element of $\fl$, contains a simple submodule~-~a result we have already proven in Theorem~\ref{cont-simple} in our setting.}

\begin{Pro}\label{localization-main}
Suppose that $M$ is a {simple} admissible    $\fl$-module of level zero and $\a\in R_0\cap R_{re}$ such that $\pm\a\in R^{in}(M)$. Then, there is  a zero-level {simple} admissible    $\fl$-module $\Omega$ such that 
\begin{itemize}
\item[(1)] $M$ is a twisted localization of $\Omega$ using a complex number,
\item[(2)] $\a\in R^{ln}(\Omega),$ 
\item[(3)] for $\b\in R_{re},$ if $\b\in R^{ln}(M)$, then $\b\in R^{ln}(\Omega)$.
\end{itemize}
\end{Pro}
\pf Pick $e_\a\in\fl^\a$ and $f_\a\in\fl^{\a}$ such that $(e_\a,f_\a,h_{\a}:=[e_\a,f_\a])$ is an $\frak{sl}_2$-triple. 
Assume $U$ is the universal enveloping algebra of $\fl$ and form $U_{-\a}$ as in \S~\ref{admissible}.
Using the same argument as in \cite[Thm.~3.35]{DG}, there is a complex number $x$ such that 
\[N:=\{v(x)\in D^x_{-\a}(M)\mid \exists n\in\bbbz^{>0}\hbox{ s.t. } e_\a^n v(x)=0\}\] is a nonzero submodule of the $\fl$-module $D^x_{-\a}(M)$; in fact,  $N$ is an admissible  $\fl$-module with $\a\in R^{ln}(N).$ Pick  a nonzero weight vector $w\in N$. The $\fl$-submodule of $D^x_{-\a}(M)$ generated by $w$ is an admissible $\fl$-module 
of level zero, so by Theorem~\ref{cont-simple}, it contains a simple submodule, say e.g., $\Omega$ and we have $\a\in R^{ln}(\Omega).$
 We have 
\begin{align*}
\Omega\stackrel{\rm U-s.m.}{\preccurlyeq }D^x_{-\a}(M)\Longrightarrow& D_{-\a}(\Omega)\stackrel{\rm U_{-\a}-s.m.}{\preccurlyeq }D_{-\a}(D^x_{-\a}(M))\simeq D^x_{-\a}(M)\\
\Longrightarrow &D_{-\a}(\Omega)\stackrel{\rm as~ U_{-\a}-module }{\hookrightarrow } D^x_{-\a}(M).
\end{align*}
But  $D^x_{-\a}(M)$ is a simple $U_{-\a}$-module and so $U_{-\a}$-modules  $D_{-\a}(\Omega)$ and $ D^x_{-\a}(M)$ are isomorphic. Therefore, $D_{-\a}(M)\simeq D_{-\a}^{-x}(\Omega)$ as $U_{-\a}$-modules (and so as $U$-modules).
 In particular, as $\pm\a\in R^{in}(M),$ we have  $U$-module isomorphisms  $M\simeq D_{-\a}(M)\simeq D_{-\a}^{-x}(\Omega)$ as  we desired. 
In particular,  as $\Omega$ is a $U_{-\a}$-module, $-\a\in R^{in}(\Omega)$ and we have  \[ \supp(M)= \supp(D^{-x}_{-\a}(\Omega))=\supp(\Omega)+\bbbz^{{\geq 0}}\a-x\a.\]  
Now, fix $\mu\in\supp(\Omega),$ then by \cite[Prop.~4.4]{you8},  
\begin{align*}
\b\in R^{in}(\Omega)&\Leftrightarrow \mu+k\b\in \supp(\Omega)~ ~~(\hbox{$\forall~k\in\bbbz^{\geq 0}$})\\
&\Rightarrow \mu-x\a+k\b\in \supp(M)~~~ (\hbox{$\forall~k\in\bbbz^{\geq 0}$})\Leftrightarrow\b\in R^{in}(M).
\end{align*}This competes the proof.
\qed

\begin{Pro}\label{twisted-local}
Suppose that $M$ is an {admissible} simple finite weight $\fl$-module. Assume $T$ is a closed symmetric subset of $R$ such that ${T_{re}}=T\cap R_{re}\neq\emptyset$ and 
$(T+\bbbz\d)\cap R\sub T.$ If    $T_{re}\sub R^{in}(M),$ then, $M$ is {of level zero} and it is an iterative twisted localization of an {admissible} zero-level simple $\fl$-module $N$  for which  
$T\cap(R_{\rm f-in}(N)\cap -R_{\rm f-in}(N))=\emptyset.$
\end{Pro}
\pf That level of $M$ is zero follows from Lemma~\ref{final}. Set $S:=T\cap R_0$ and  recall that  $2(R_{re}\cap R_1)\sub R_0.$ This in turn implies that   $S\cap R_{re}\neq \emptyset.$
One knows \[ \dot S:=\{\dot\b\in \dot R\mid (\dot\b+\bbbz\d)\cap S\neq \emptyset\},\]
 is a finite root system.  Fix a set $\dot S^+=\{\dot\a_1,\ldots,\dot\a_\ell\}$ of positive roots of $\dot S$ with respect to a fixed base of $\dot S$ and for each 
$1\leq i\leq \ell,$ fix  a root $\a_i\in S\cap(\dot \a_i+\bbbz\d).$ We have $\a_1, -\a_1\in  {R^{in}(M)}$ and so, by Proposition~\ref{localization-main}, there is a complex number $x$ such that  $M\simeq D^{x}_{-\a_1}(\Omega)$ for a simple admissible  $\fl$-module $\Omega$ for which $\a_1\in R^{ln}(\Omega)$, that is, $(\pm\dot\a_1+\bbbz\d)\cap R^{ln}(\Omega)\neq\emptyset$. Suppose that  $1\leq t\leq \ell$ and  
$(\pm\dot\a_i+\bbbz\d)\cap R^{ln}(\Omega)\neq\emptyset$ for $1\leq i\leq t.$ If $t=\ell,$ then, we are done,
 otherwise, since $\pm\a_{t+1}\in R^{in}(\Omega),$ as above there are a  complex number $y$ and a simple admissible  module $\Omega'$ such that $\Omega=D^{y}_{-\a_{t+1}}(\Omega')$ with   $\a_{t+1}\in R^{ln}(\Omega');$ in particular, it follows form   part~(3) of Proposition~\ref{localization-main} that    $(\pm\dot\a_i+\bbbz\d)\cap R^{ln}(\Omega')\neq\emptyset$ for all $1\leq i\leq t+1.$  Repeating this process, we will get finitely many complex numbers $y_1,\ldots,y_s,$  indices $1\leq i_1,\ldots,i_s\leq \ell$ and a zero-level simple admissible   module $N$ such that $M\simeq D_{-\a_{i_s}}^{y_s}\cdots D_{-\a_{i_1}}^{y_1}(N)$ and $(\pm\dot\a_i+\bbbz\d)\cap R^{ln}(N)\neq\emptyset$ for all $1\leq i\leq \ell.$ This completes the proof. \qed
\medskip

Proposition~\ref{twisted-local} together with the results of \cite{DKYY, DKY,KRY, you10} give the following theorem which in turn completes the characterization problem of admissible  $\fl$-modules due to \cite{DMP, KRY,  FGG, G, H, Mat, you8, you11}. To state this theorem, we  first recall the definition of  the quadratic Lie superalgebra from \cite{KRY}; it  is the superspace \[\cQ=\underbrace{s^2\bbbc[s^{\pm4}]\op t^2\bbbc[t^{\pm4}]}_{\cQ_0}\op \underbrace{t\bbbc[t^{\pm4}]\op t^{-1}\bbbc[t^{\pm4}]}_{\cQ_1}\] together 
with the bracket 
\begin{equation*}\label{bracket}\begin{array}{lll}
~[\cQ_0,\cQ_0]=[t^2\bbbc[t^{\pm4}], \cQ_1]:=\{0\},& ~[t^{4k+1},t^{4k'-1}]:=0,&
~[t^{4k+1},t^{4k'+1}]:=t^{4(k+k')+2},\\
~[t^{4k-1},t^{4k'-1}]:=-t^{4(k+k')-2},&
~[s^{4k-2},t^{4k'+1}]:=t^{4(k+k')-1},&
~[s^{4k+2},t^{4k'-1}]:=t^{4(k+k')+1}.
\end{array}
\end{equation*}
In \cite{KRY}, we characterize $\bbbz$-graded simple modules with uniformly bounded dimensions of homogeneous spaces, over $A\op\cQ$ where $A$ is a $\bbbz$-graded abelian Lie algebra. We recall that affine Lie superalgebra $\fl$ has a natural $\bbbz$-grading and that the  standard Cartan subalgebra $\fh$ of $\fl$ contains a vector $d$ acting  as a degree derivation on $\fl.$ In \cite{KRY}, we also give a complete   characterization of finite $\fh$-weight modules over   $(A\op\cQ)\rtimes \fh$ where the action of $\fh$ on $A\op\cQ$ is just the action of $d$ on $A\op\cQ$ as a degree derivation.
\begin{Thm}\label{twisted-local-2}
Let  $M$ be a simple {admissible}  $\fl$-module  such that $R_{\rm f-in}(M)\cap-R_{\rm f-in}(M)\neq\emptyset.$ Then, the level of  $M$ is zero and $M$ is an iterative twisted localization of a zero-level simple {admissible} $\fl$-module $N$  for which  
$R_{\rm f-in}(N)\cap -R_{\rm f-in}(N)=\emptyset.$ In particular, there are  $\b_1,\ldots,\b_s\in R_{re}\cap R_0$, finitely many complex numbers $y_1,\ldots,y_s,$  a parabolic subset $P$ of $R$  and a finite $\fh$-weight   module $\Omega$ over  $\LL^\circ=\Bigop{\a\in P\cap-P}{}\LL^\a$ such that 
\[M\simeq D_{\b_s}^{y_s}\cdots D_{\b_1}^{y_1}{\rm Ind}_P(\Omega);\] moreover, 
$\Omega$ and $\LL^\circ$ satisfy one of the  following:
\begin{itemize}
\item  $\Omega$ is  cuspidal and $\fl^\circ$ is a  direct sum of finitely many finite dimensional  Lie superalgebras which are either a reductive Lie algebra or a basic classical simple Lie superalgebra.
\item  $\fl^\circ=\Bigop{k\in\bbbz}{}\LL^{k\d}$ which in turn is a semi direct product  $H\rtimes \fh$ in which $H$ is
\begin{itemize}
\item  a $\bbbz$-graded abelian Lie algebra if $\fl\neq A(2k,2l)^{(4)},$ 
\item  $A\op\cQ$ where  $A$ is a $\bbbz$-graded abelian Lie algebra if $\fl=A(2k,2l)^{(4)}$.
\end{itemize}
 \end{itemize}
\end{Thm}

We conclude the paper by determining which affine Lie superalgebras $\mathscr{A}$, with the corresponding root system 
$R$ admit finite weight modules 
$M$ satisfying $R_{re}=R^{in}(M).$ To this end, we require the following lemma:

\begin{lem}\label{above}
Suppose that $\fk$ is a finite dimensional reductive Lie algebra with center $Z$ and  semisimple part $[\fk,\fk]=\Bigop{i=1}{n}\fg_i$ in which each $\fg_i$ is a simple Lie algebra with Cartan subalgebra $\fh_i$. Set $\fh:=\Bigop{i=1}{n}\fh_i\op Z$ and assume $M$ is an admissible  finite  $\fh$-weight $\fk$-module.  Fix  $1\leq i\leq n$ and  suppose that there is $\lam\in \supp(M)$ such that $\fg_i$-submodule generated by $M^{\lam}$ is infinite dimensional, then  $\fg_i$ is either of type $A$ or  of type $C$.
\end{lem}
\pf Suppose $i$ is as in the statement and set  \[\ft:=Z+\sum_{i\neq j=1}^n\fh_j  ~~\hbox{ as well as } ~~ \fl_i:=\fg_i+\fh=\fg_i\op\ft.\] We consider each $\a\in\fh_i^*$  as a functional on $\fh$ by zero value on $\ft.$ Recall  $\lam$ from the statement and set $\aa:=\{\mu\in \fh^*\mid \mu(h)=\lam(h)~~(\forall h\in \ft)\}.$  Then \[N:=\Bigop{\mu\in \aa}{}M^\mu\] is an {admissible}  finite $\fh_i$-weight $\fg_i$-module  and the $\fg_i$-submodule generated by $M^{\lam}$ is an 
infinite dimensional {admissible} finite $\fh_i$-weight $\fg_i$-module; in particular,  it is of finite length by Lemma~3.3 of \cite{Mat} and so it has a simple infinite dimensional subfactor.
Therefore, by \cite[Prop.~1.4]{BBL}, $\fg_i$ is either of type $A$ or of type $C.$\qed

\begin{Pro}
Suppose that $M$ is a  finite weight module over an affine Lie superalgebra $\mathscr{A}$ with $\mathscr{A}\neq \{0\}.$ Suppose $R=R_0\cup R_1$ is its root system and assume  $R^{in}(M)=R_{re}.$ Then $\mathscr{A}$ can be only   of types  $A(\ell-1,\ell-1)^{(1)},$
$A(k-1,\ell-1)^{(1)},$ $B(0,\ell)^{(1)},$ $D(2,1;\lam)^{(1)}$ 
$(\lam\neq 0,-1)$,  $D(1,\ell)^{(1)},$ $A(0,2\ell-1)^{(2)},$ $A(1,2\ell-1)^{(2)},$ $A(0,2\ell)^{(4)}$ and $D(1,\ell)^{(2)}.$
\end{Pro}
\pf  Set
$\dot S:=\dot R\cap R_0$ and $\fk:=\Bigop{\a\in \dot S}{}\mathscr{A}^\a.$ Then $\fk$ is a reductive Lie algebra. Fix $\lam\in\supp(M).$ Then, the $\fk$-submodule $N$ generated by $M^\lam$  is a finite weight module over $\fk.$ Since all nonzero root vectors corresponding to real roots act injectively, the dimensions of weight spaces of $N$ are uniformly bounded and also for any simple constituent $\fg$ of $\fk$, the $\fg$-module generated by any nonzero weight vector is infinite dimensional.
So, by Lemma~\ref{above},  each simple constituent of $\fk$ is either of type $A$ or of type $C.$ So, $\mathscr{A}$ is of one of the types $A(\ell-1,\ell-1)^{(1)},$
$A(k-1,\ell-1)^{(1)},$ $B(0,\ell)^{(1)},$ $D(2,1;\lam)^{(1)}$ 
$(\lam\neq 0,-1)$,  $D(1,\ell)^{(1)},$ $A(0,2\ell-1)^{(2)},$ $A(1,2\ell-1)^{(2)},$ $A(0,2\ell)^{(4)}$ and $D(1,\ell)^{(2)}.$
\qed

\end{document}